\documentclass[10pt]{article}
\usepackage{mathrsfs}
\usepackage{mathrsfs}
\usepackage{mathrsfs}
\usepackage{cite}
\usepackage{amsfonts}
\usepackage{amssymb}
\usepackage{amsmath}
\usepackage{amsthm}
\usepackage{dsfont}
\usepackage[colorlinks,
            linkcolor=black,
            anchorcolor=black,
            citecolor=black
            ]{hyperref}
\theoremstyle{plain}
\newtheorem{thm}{Theorem}[section]

\newtheorem{lem}[thm]{Lemma}

\theoremstyle{definition}
\newtheorem{defn}{Definition}[section]

\theoremstyle{remark}

\allowdisplaybreaks[4]

\usepackage[pagewise]{lineno}
\begin{document}
\title{{\Large\bf {
Second order necessary conditions for quantum stochastic optimal control problems}}
\thanks{This work is supported by National Natural Science Foundation of China (No.11471189, No.11871308).}}
\author{{\normalsize  Penghui Wang, \ Shan Wang{\thanks{E-mail addresses: phwang@sdu.edu.cn(P.Wang), swangmath@hebtu.edu.cn(S.Wang).}}  \,\,\,\,  } \\
{\normalsize School of Mathematics, Shandong University,} {\normalsize Jinan, 250100, China}\\
{\normalsize School of Mathematics Sciences, Hebei Normal University,}{ \normalsize Shijiazhuang, 050016, China}}
\date{}
\maketitle
\begin{minipage}{14cm} {\bf Abstract }
This paper aims to establish second order necessary conditions for optimal control in quantum stochastic systems. We employ a variational approach, analogous to methods in classical stochastic control, to analyze systems governed by quantum stochastic differential equations driven by fermionic Brownian motion, where the control enters both the drift and diffusion terms. 
This result provides a theoretical foundation for further exploration of optimization problems and their practical applications in the field of quantum stochastic control.

\end{minipage}
 \maketitle
\numberwithin{equation}{section}
\newtheorem{theorem}{Theorem}[section]
\newtheorem{lemma}[theorem]{Lemma}
\newtheorem{proposition}[theorem]{Proposition}
\newtheorem{corollary}[theorem]{Corollary}
\newtheorem{remark}[theorem]{Remark}


\section{Introduction}\label{Intro}
\indent\indent
Let $(\Lambda(\mathscr{H}), \mathscr{C}, m)$ be a quantum (noncommutative) probability space \cite{B.S.W.1,B.S.W.2,P.book,P.X,W.F,W.W-1,W.W-2}. The Hilbert space $\Lambda(\mathscr{H})$ is the anti-symmetric Fock space over the separable complex Hilbert space $\mathscr{H}=L^2(\mathbb R^+)$. The fermion field $\Psi(z)$ is defined on $\Lambda(\mathscr{H})$ by
\begin{equation*}
\Psi(z):=C(z)+A(Jz),\ z\in \mathscr{H},
\end{equation*}
where $C(z)$ and $A(z)$ are the creation operator and the annihilation operator, respectively. Additionally, $J$ is the complex conjugation operator on $L^2(\mathbb{R}^+,ds)$ (i.e., $J f=\overline{f}$). 
For any $t\in[0,\infty)$, $\mathscr{C}_t$ is a W*-subalgebra generated by $$\{\Psi(v): v\in L^2(\mathbb{R}^+,ds)\ \textrm{and}\ \textrm{ess supp}\ v\subset [0,t]\},$$ and $\mathscr{C}$ is von Neumann algebra generated by the increasing sequence of von Neumann subalgebras $\{\mathscr{C}_t\}_{t\in\mathbb{R}^+}$.
For any $t\in[0,\infty)$,  the fermion Brownian motion $W(t)$ is defined by
\begin{equation}\label{Fermion Brownian motion}
W(t):=\Psi(\chi_{[0,t]})=\mathscr{A}^*(\chi_{[0,t]})+\mathscr{A}(J\chi_{[0,t]}), 
\end{equation}
which adapted to the family $\{\mathscr{C}_t: t\in \mathbb{R}^+\}$ \cite{B.S.W.1,B.S.W.2,P.X}. For the Fock vacauum $\mathds{1}\in\Lambda(\mathscr{H})$, define a state on $\mathscr{C}$ by 
\begin{equation}\label{m-state}
m(\cdot):=\langle\mathds{1}, \cdot\mathds{1}\rangle.
\end{equation}
Clearly, $m$ is  normal and  faithful. For any $ p\in[1,\infty)$, the noncommutative $L^p$-norm on $\mathscr{C}$ is defined by
 $$\|f\|_{p}:=m(|f|^p)^{\frac{1}{p}}=\langle\mathds{1}, |f|^p\mathds{1}\rangle^{\frac{1}{p}},$$
where $|f|=(f^*f)^{\frac{1}{2}}$. The noncommutative $L^p$-space $L^p(\mathscr{C},m)$ is the completion of $(\mathscr{C}, \|\cdot\|_p)$,  abbreviated as $L^p(\mathscr{C})$.

Let $\mathcal{X}$ be a Banach space. For any given $T>0$, denoted by $C([0,T];\mathcal{X})$ the Banach space of all $\mathcal{X}$-valued continuous functions on $[0,T]$. For any $r\in[1,\infty)$,
let $L^r(0,T;\mathcal{X})$ be the Banach space of all  $\mathcal{X}$-valued functions that are $r$th power Lebesgue integrable on $[0,T]$.
In addition, let
\begin{equation*}
C_\mathbb{A}([0,T];L^p(\mathscr{C})):=\left\{f:f\in C([0,T];L^p(\mathscr{C}));\ f(t)\in L^p(\mathscr{C}_t),\  \textrm{a.e.}\ t\in [0,T]\right\},
\end{equation*}
and
\begin{equation*}
L^r_\mathbb{A}(0,T;L^p(\mathscr{C})):=\left\{f:f\in L^r(0,T;L^p(\mathscr{C}));\ f(t)\in L^p(\mathscr{C}_t),\  \textrm{a.e.}\ t\in [0,T]\right\}.
\end{equation*}
Let $\mathcal{X}_1$, $\mathcal{X}_2$ and $\mathcal{X}_3$ be three Banach spaces. Denote by $\mathcal{L}(\mathcal{X}_1;\mathcal{X}_2)$ the Banach space of all bounded linear operators from $\mathcal{X}_1$ to $\mathcal{X}_2$  with the usual operator norm, and denote $\mathcal{L}(\mathcal{X}_1;\mathcal{X}_1)$ simply as $\mathcal{L}(\mathcal{X}_1)$. Denote by $\mathcal{L}(\mathcal{X}_1,\mathcal{X}_2;\mathcal{X}_3)$ the Banach space of all bounded linear operators from $\mathcal{X}_1\times \mathcal{X}_2$ to $\mathcal{X}_3$  with the usual operator norm.

Let $\mathcal{H}$ be a separable Hilbert space with norm $\|\cdot\|_{\mathcal{H}}$ and inner product $\langle\cdot,\cdot\rangle_{\mathcal{H}}$. 
Assume that $U$ is a closed convex subset of $\mathcal{H}$, and  put
\begin{equation*}
  \mathcal{U}^\beta[0,T]:=\left\{u:[0,T]\to U\hspace{0.5mm}|\hspace{1mm} \|u\|_{ L^\beta(0,T;\mathcal{H})}<\infty\right\}.
\end{equation*}
This paper investigates the second order necessary condition for optimal control of the following quantum stochastic control system in noncommutative space $L^2(\mathscr{C})$:
\begin{equation}\label{FSDE-control}
 \left\{
 \begin{aligned}
  &dx(t)=D(t,x(t),u(t))dt+F(t,x(t),u(t))dW(t)+dW(t)G(t,x(t),u(t)),\quad \textrm{in}\ [t_0,T],\\
  & x(t_0)=x_0,
  \end{aligned}
  \right.
\end{equation}
with the cost functional
\begin{equation}\label{Cost functional introduced}
  \mathcal{J}(u(\cdot))=\int_{t_0}^TL(t,x(t),u(t))dt+g(x(T)), \quad u(\cdot)\in \mathcal{U}^\beta[t_0,T].
\end{equation}
Here the initial condition $x_0\in L^2(\mathscr{C}_{t_0})$, the maps $D(\cdot,\cdot,\cdot), F(\cdot,\cdot,\cdot), G(\cdot,\cdot,\cdot):[t_0,T]\times L^2(\mathscr{C})\times U \to L^2(\mathscr{C})$ are adapted, continuous operator-valued functions, $W(t)$ is fermion Brownian motion defined by \eqref{Fermion Brownian motion} of quantum (noncommutative) probability space $(\Lambda(\mathscr{H}), \mathscr{C}, m)$, the cost density
$L(\cdot,\cdot,\cdot):[t_0,T]\times L^2(\mathscr{C})\times U \to \mathbb{R}$, and the terminal cost $g(\cdot): L^2(\mathscr{C}_T)\to\mathbb{R}$.

The optimal control problem for the quantum stochastic system \eqref{FSDE-control} is stated as follows:\\
\textbf{Problem (QOP).} Find a $\bar{u}(\cdot)\in \mathcal{U}^\beta[t_0,T]$ such that
\begin{equation}\label{QOP}
\mathcal{J}(\bar{u}(\cdot))=\inf_{u(\cdot)\in\mathcal{U}^\beta[t_0,T]}\mathcal{J}(u(\cdot)).
\end{equation}
Any $\bar{u}(\cdot)$ satisfying \eqref{QOP} is called an \textit{optimal control}. The corresponding $\bar{x}(\cdot)$ and $(\bar{x}(\cdot), \bar{u}(\cdot))$ are called an \textit{optimal state}  and \textit{optimal pair} of quantum control systems, respectively.

Optimal control problems are inherently a class of optimization problems with specific structures, for which Pontryagin’s maximum principle provides first-order necessary conditions. However, these first-order conditions only capture essential characteristics that an optimal solution must satisfy; they neither guarantee optimality nor effectively distinguish between local and global optima. Therefore, establishing corresponding second-order necessary conditions becomes an indispensable theoretical step in order to accurately identify true optimal solutions from candidate solutions that satisfy the first-order conditions.
In deterministic systems, research on second-order conditions can be traced back to the early development of Pontryagin’s principle \cite{B.J,G.}, and the corresponding theoretical framework is now well-established \cite{W.,F.L-2021,F.T}. In contrast, the development of second-order conditions in stochastic optimal control has encountered significantly greater challenges. Fundamental difficulties arise especially in cases with non-convex control domains and control-dependent diffusion terms, such as the loss of orders when estimating the variational terms of stochastic integral.
Research in this direction started relatively late. Although early studies explored certain special cases, a systematic theoretical breakthrough did not emerge until the late 1990s. In 1997,  Mammadov and Bashirov made a key advance in \cite{M.B} by establishing rigorous second-order necessary conditions for systems with uncontrolled diffusion terms for the first time, thereby laying an important foundation for subsequent research. Since then, the theory has evolved rapidly along several branches. For instance, Tang \cite{T} established a second-order necessary condition for singular optimal controls within the framework of Pontryagin's maximum principle, under the assumption that the diffusion term is independent of the control variable.  
Zhang and Zhang \cite{Z.Z-2015} first employed Malliavin calculus to established the pointwise second-order maximum principle for singular optimal controls in the classical sense with the convex control constraint. And then,  L\"{u}, Zhang and Zhang \cite{L.Z.Z} addressed the second optimality condition for optimal control problems of stochastic evolution equations in infinite dimension. Further details and extensions of these results can be found in the related works \cite{B.S,F.L-2020,F.Z.Z-2017,F.Z.Z-2018,Z.Z-2017,Z.Z-2018}.

Inspired by above research, we investigate the second order necessary conditions for optimal control problems of quantum stochastic system driven by fermion Brownian motion in this paper, which is an important class of physical models used to describe quantum optical systems \cite{G.Z}.
In \cite{W.W-1}, we have obtained the Pontryagin-type maximum principle for \textbf{Problem (QOP)}, which is a first order necessary condition for quantum optimal controls. 
However, due to the influence of quantum noise, the control problems of infinite-dimensional quantum stochastic systems become complex under conditions of nonlinearity, constraints, or high uncertainty, potentially leading to singular control problems. Therefore, when solving such problems, it is essential to consider the second order necessary conditions for optimal control to ensure the stability and optimality of the control strategy. 
 
To achieve the desired result, we have undertaken the following efforts.
First, based on the isometric property of stochastic integrals in non-commutative spaces, we derive the necessary estimates. 
Next, we introduce the parity operator $\Upsilon$ to address the non-commutativity between operators, which enables us to construct the required variational equations and adjoint equations. Additionally, we propose higher differentiability conditions for the cost functional, allowing us to obtain the corresponding Taylor expansion \cite{Z.Z-2015,Z.Z,Z.Z-2017,Z.Z-2018}.
Finally, using the estimates from the variational equations, we obtain the main results by combining the Fermion It\^{o}'s formula and the relaxed transposition solutions to the second order adjoint equations. 

The rest of this paper is organized as follows. In Section \ref{Preliminary}, we provide some useful
estimates corresponding to the control system and present some results for quantum stochastic differential equations. In Section \ref{main}, we  establish the integral-type second order necessary optimality conditions for quantum stochastic optimal controls.

\section{Preliminaries}\label{Preliminary}
\indent\indent
In this section, we provide some  preliminaries and necessary assumptions which would be useful in the sequel. 
For the following stochastic control systems studied in this paper 
\begin{equation*}
 \left\{
 \begin{aligned}
  &dx(t)=D(t,x(t),u(t))dt+F(t,x(t),u(t))dW(t)+dW(t)G(t,x(t),u(t)),\quad \textrm{in}\ [t_0,T],\\
  & x(t_0)=x_0,
  \end{aligned}
  \right.
\end{equation*}
we then present the corresponding assumptions as follows:
\begin{description}
  \item[(H1)]  The maps $D(\cdot,\cdot,\cdot), F(\cdot,\cdot,\cdot), G(\cdot,\cdot,\cdot):[t_0,T]\times L^2(\mathscr{C})\times U\to L^2(\mathscr{C})$ are adapted. For $\Phi=D,F,G$, there exists a constant $\mathcal{C}>0$ such that for any 
      $x, \hat{x}\in L^2(\mathscr{C})$ and $u, \hat{u}\in U$,
\begin{equation*}
\left\{
\begin{aligned}
&\|\Phi(t,x,u)-\Phi(t,\hat{x},\hat{u})\|_2\leq \mathcal{C}(\|x-\hat{x}\|_2+\|u-\hat{u}\|_{\mathcal{H}}),\\
&\|\Phi(t,0,0)\|_2\leq \mathcal{C}.
\end{aligned}
\right.
\end{equation*}
  \item[(H2)] The maps $L:[t_0,T]\times L^2(\mathscr{C})\times U\to \mathbb{R}$ and $g:L^2(\mathscr{C}_T)\to \mathbb{R}$ are measurable, and there exists a constant $\mathcal{C}>0$ 
   such that for any $t\in[t_0, T]$, $x, \hat{x}\in L^2(\mathscr{C})$ and $u,\hat{u}\in U$
      \begin{equation*}
\left\{
\begin{aligned}
&|L(t,x,u)-L(t,\hat{x},\hat{u})|\leq \mathcal{C}(\|x-\hat{x}\|_2+\|u-\hat{u}\|_{\mathcal{H}}),\\
&|g(x)-g(\hat{x})|\leq \mathcal{C}\|x-\hat{x}\|_2,\\
&|L(t,0,u)|+|g(0)|\leq \mathcal{C}.
\end{aligned}
\right.
\end{equation*}
 \item[(H3)] The maps $D,\ F,\  G$  are second order Fréchet differentiable with respect to $x$ and $u$. For $\Phi=D,F,G$, the operators $\Phi_x(t,x,u)\in \mathcal{L}(L^{2}(\mathscr{C}))$, $\Phi_u(t,x,u)\in\mathcal{L}(\mathcal{H}; L^{2}(\mathscr{C}))$, $\Phi_{xu}(t,x,u)\in \mathcal{L}(L^2(\mathscr{C}),\mathcal{H};L^2(\mathscr{C}))$, $\Phi_{xx}(t,x,u)\in \mathcal{L}(L^2(\mathscr{C}),L^2(\mathscr{C});L^2(\mathscr{C}))$ and $\Phi_{uu}(t,x,u)\in \mathcal{L}(\mathcal{H},\mathcal{H};$ $L^2(\mathscr{C}))$ are continuous, and there exists a constant $\mathcal{C}$ such that, for any $(t,x,u)\in [t_0,T]\times L^2(\mathscr{C})\times U$,
     \begin{equation*}
     \left\{
     \begin{aligned}
        & \|\Phi_x(t,x,u)\|_{\mathcal{L}(L^{2}(\mathscr{C}))}+\|\Phi_u(t,x,u)\|_{\mathcal{L}(\mathcal{H};L^{2}(\mathscr{C}))}\leq \mathcal{C}(1+\|x\|_2+\|u\|_\mathcal{H}), \\
       &\|\Phi_{xx}(t,x,u)\|_{\mathcal{L}(L^2(\mathscr{C}),L^2(\mathscr{C}); L^{2}(\mathscr{C}))}+ \|\Phi_{uu}(t,x,u)\|_{\mathcal{L}(\mathcal{H},\mathcal{H}; L^{2}(\mathscr{C}))}\\
      &\qquad\qquad\qquad\qquad\qquad\qquad +\|\Phi_{xu}(t,x,u)\|_{\mathcal{L}(L^{2}(\mathscr{C}),\mathcal{H}; L^{2}(\mathscr{C}))}\leq \mathcal{C}.
     \end{aligned}
     \right.
     \end{equation*}
  \item[(H4)] For any $t\in[t_0,T]$, the map  $L(t,\cdot,\cdot):L^2(\mathscr{C})\times U\to\mathbb{R}$  is second order Fréchet differentiable with respect to $x$ and $u$, and $g(\cdot):L^2(\mathscr{C}_T)\to\mathbb{R}$ is second order Fréchet differentiable  with respect to $x$. Moreover,  there exists a constant $\mathcal{C}$ such that for a.e. $t\in[t_0,T]$ and any $(x,u)\in L^2(\mathscr{C})\times U$,
       \begin{equation*}
        \left\{
        \begin{aligned}
        &\|L_{x}(t,x,u)\|_2+\|g_{x}(x)\|_{2}+\|L_{u}(t,x,u)\|_\mathcal{H}\leq \mathcal{C}(1+\|x\|_2+\|u\|_\mathcal{H})\\
&\|L_{xx}(t,x,u)\|_{\mathcal{L}(L^2(\mathscr{C}))}+\|g_{xx}(t,x,u)\|_{\mathcal{L}(L^{2}(\mathscr{C}))}\leq \mathcal{C},\\
 & \|L_{xu}(t,x,u)\|_{\mathcal{L}(L^2(\mathscr{C});\mathcal{H})}+\|L_{uu}(t,x,u)\|_{\mathcal{L}(\mathcal{H})}\leq \mathcal{C}.
        \end{aligned}
        \right.
      \end{equation*}
\end{description}

Recall that for any $p\in[1,\infty)$, the grading automorphism $\Upsilon$ on $L^p(\mathscr{C})$, as discussed in \cite{P.X}, is uniquely determined by 
\begin{equation*} 
\Upsilon(\Psi(v_1)\Psi(v_2)\cdots \Psi(v_n)) = (-1)^n \Psi(v_1)\Psi(v_2)\cdots \Psi(v_n), \quad v_i \in \mathscr{H}, \quad 1 \leq i \leq n. 
\end{equation*} 
Furthermore, $\Upsilon$ is referred to as the parity operator. It is evident that $\Upsilon$ is unitary and satisfies $\Upsilon^2=I$.

\begin{defn} \cite{P.X,B.S.W.1,B.S.W.2}
 An element $f \in L^p(\mathscr{C})$ is said to be even (resp.
odd) if $\Upsilon(f) = f$ (resp. $\Upsilon(f)=-f$).
\end{defn}

Then, one obtains the following result, which is derived from \cite{B.S.W.2}.
\begin{lem}\textnormal{\cite[Theorem 2.1]{B.S.W.2}}\label{estimate of x}
Assuming that $\textbf{(H1)}$ holds, the equation \eqref{FSDE-control} admits a unique solution $x(\cdot)\in C_\mathbb{A}([t_0,T];L^2(\mathscr{C}))$, and it holds that
\begin{equation*}
\|x(\cdot)\|_{C_\mathbb{A}([t_0,T];L^2(\mathscr{C}))}\leq \mathcal{C}\left\{1+\|x_0\|_2+\|u\|_{L^2(t_0,T;\mathcal{H})}\right\}.
\end{equation*}
\end{lem}
 Assume that $\bar{u}(\cdot)\in \mathcal{U}^4[0,T]$ is an optimal control and $\bar{x}(\cdot)$ is corresponding optimal state. 
 Let us introduce some notations. For any $\Phi=D,F,G,L$,  write
\begin{equation*}
  \left\{
\begin{aligned}
&\Phi_x(t):=\Phi_x(t,\bar{x}(t), \bar{u}(t)),\\
&\Phi_u(t):=\Phi_u(t,\bar{x}(t), \bar{u}(t)), \\
&\Phi_{xx}(t):=\Phi_{xx}(t,\bar{x}(t), \bar{u}(t)),\\
&\Phi_{uu}(t):=\Phi_{uu}(t,\bar{x}(t), \bar{u}(t)),\\
&\Phi_{xu}(t):=\Phi_{xu}(t,\bar{x}(t), \bar{u}(t)). 
\end{aligned}
  \right.
\end{equation*}
To establish the second order necessary condition for \textbf{Problem (QOP)}, 
 we introduce the following adjoint equations:
\begin{equation}\label{BQSDE-Y}
\left\{
\begin{aligned}
dy(t)=&-\{D_x(t)^* y(t)+(F_x(t)+\Upsilon G_x(t))^*Y(t)-L_x(t)\}dt+Y(t)dW(t),\  {\rm{in}}\ [0,T),\\
y(T)=&-g_x(\overline{x}(T)),
\end{aligned}
\right.
\end{equation}
and 
\begin{equation}\label{BQSDE-P}
\left\{
\begin{aligned}
dP(t)=&-\Big\{D_{x}(t)^*P(t)+P(t)D_{x}(t)+(F_x(t)+\Upsilon G_x(t))^*Q(t)\Upsilon+Q(t)\Upsilon(F_x(t)+\Upsilon G_x(t))\\
&\hspace{8mm}+(F_x(t)+\Upsilon G_x(t))^*P(t)(F_x(t)+\Upsilon G_x(t))+\mathbb{H}_{xx}(t)\Big\}dt+Q(t)dW(t),\ {\rm{in}}\ [0,T),\\
P(T)=&-g_{xx}(\overline{x}(T)),
\end{aligned}
\right.
\end{equation}
where $y(T)\in L^2(\mathscr{C}_T)$, $P(T)\in \mathcal{L}(L^2(\mathscr{C}_T))$. The Hamilton function $\mathbb{H}(\cdot,\cdot,\cdot,\cdot,\cdot)$ is defined by
\begin{align*}
\mathbb{H}(t,x,u,y,Y)&:= \langle y, D(t,x,u) \rangle+\langle Y, F(t,x,u)+\Upsilon G(t,x,u) \rangle-L(t,x,u),\\
 &(t,x,u,y,Y)\in [0,T]\times L^2(\mathscr{C})\times U \times  L^2(\mathscr{C})\times  L^2(\mathscr{C}),
\end{align*}
and simplify the notations
\begin{equation*}
\left\{
  \begin{aligned}
  &\mathbb{H}_x(t):=\mathbb{H}_x(t,x,u,y,Y), \\
  &\mathbb{H}_u(t):=\mathbb{H}_u(t,x,u,y,Y), \\
  & \mathbb{H}_{xx}(t):=\mathbb{H}_{xx}(t,x,u,y,Y), \\
  &  \mathbb{H}_{uu}(t):=\mathbb{H}_{uu}(t,x,u,y,Y), \\
  &  \mathbb{H}_{xu}(t):=\mathbb{H}_{xu}(t,x,u,y,Y).
\end{aligned}
\right.
\end{equation*}

The following result related to the equation \eqref{BQSDE-Y} is required and is found in \cite{W.W-1}. 
\begin{lem}\textnormal{\cite[Theorem 4.1]{W.W-1}}
Under the assumptions \textbf{(H3)} and \textbf{(H4)}, the equation \eqref{BQSDE-Y} admits a unique solution $(y(\cdot),Y(\cdot))\in C_\mathbb{A}([0,T];L^2(\mathscr{C}))\times  L^2_\mathbb{A}(0,T;L^2(\mathscr{C}))$, and
\begin{equation}
\|(y(\cdot),Y(\cdot))\|_{C_\mathbb{A}([0,T];L^2(\mathscr{C}))\times  L^2_\mathbb{A}(0,T;L^2(\mathscr{C}))}\leq \mathcal{C}\left(\|g_x(\overline{x}(T))\|_2+\int_{0}^{T}\|L_x(t)\|_2dt\right).
\end{equation}
\end{lem}
To provide the definition of the solution to the equation \eqref{BQSDE-P}, let us introduce the following quantum stochastic differential equations (for short QSDEs):
\begin{equation}\label{qsde1 to second adjoint equation}
  \left\{
  \begin{aligned}
  &d\phi_1(s)=\{D_{x}(s)\phi_1(s)+\mu_1(s)\}ds+\{(F_x(s)+\Upsilon G_x(s))\phi_1(s)+\nu_1(s)\}dW(s),\ \textrm{in}\ [t,T],\\
  &\phi_1(t)=\zeta_1,
  \end{aligned}
  \right.
\end{equation}
and 
\begin{equation}\label{qsde2 to second adjoint equation}
  \left\{
  \begin{aligned}
  &d\phi_2(s)=\{D_{x}(s)\phi_2(s)+\mu_2(s)\}dt+\{(F_x(s)+\Upsilon G_x(s))\phi_2(s)+\nu_2(s)\}dW(s),\ \textrm{in}\ [t,T],\\
  &\phi_2(t)=\zeta_2,
  \end{aligned}
  \right.
\end{equation}
where $\zeta_1,\zeta_2\in L^2(\mathscr{C}_{t})$ and $\mu_1,\mu_2(\cdot),\nu_2(\cdot),\nu_2(\cdot)\in L^2_\mathbb{A}(t,T;L^2(\mathscr{C}))$.
Write
$$\mathfrak{H}_t:=L^2(\mathscr{C}_{t})\times L^2_\mathbb{A}(t,T;L^2(\mathscr{C}))\times L^2_\mathbb{A}(t,T;L^2(\mathscr{C})), \quad t\in[0,T],$$
and
\begin{equation*}
 \mathcal{Q}[0,T]:=\Big\{\left(Q^{(\cdot)},\widehat{Q}^{(\cdot)}\right);\ Q^{(t)}, \widehat{Q}^{(t)}\in\mathcal{L}(\mathfrak{H}_t;L^2_\mathbb{A}(t,T;L^2(\mathscr{C}))), \text{and}\ Q^{(t)}(0,0,\cdot)^*=\widehat{Q}^{(t)}(0,0,\cdot),\ t\in[0,T]\Big\}.
\end{equation*}
\begin{defn}\cite[Definition 1.1]{W.W-2}\label{definition of the relaxed tran solution}
We call $\left(P(\cdot),Q^{(\cdot)},\widehat{Q}^{(\cdot)}\right) \in C_\mathbb{A}([0,T];\mathcal{L}(L^2(\mathscr{C})))\times \mathcal{Q}[0,T]$ a relaxed transposition solution to \eqref{BQSDE-P} if for any $t\in[0,T]$, $\zeta_1,\zeta_2\in L^2(\mathscr{C}_t)$ and $\mu_1 $, $\mu_2 $, $\nu_2 $, $\nu_2 \in L^2_\mathbb{A}([t,T];L^2(\mathscr{C}))$, it holds that
\begin{equation*}
  \begin{aligned}
  &\langle P_T\phi_2(T),\phi_1(T)\rangle+\int_{t}^{T}\langle \mathbb{H}_{xx}(s)\phi_2(s),\phi_1(s)\rangle ds\\
  &=\langle P(t)\zeta_2,\zeta_1\rangle+\int_{t}^{T}\langle P(s)\mu_2(s),\phi_1(s) \rangle  ds+\int_{t}^{T}\langle P(s)\phi_2(s), \mu_1(s)\rangle  ds\\
  &\hspace{4.7mm}+\int_{t}^{T}\langle P(s)(F_x(s)+\Upsilon G_x(s))\phi_2(s), \nu_1(s)\rangle dt+\int_{t}^{T}\langle P(s)\nu_2(s), (F_x(s)+\Upsilon G_x(s))\phi_1(s)+\nu_1(s)\rangle ds\\
  &\hspace{4.7mm}+\int_{t}^{T}\langle Q^{(t)}(\zeta_2,\mu_2,\nu_2)(s),\nu_1(s)\rangle ds+\int_{t}^{T} \langle\nu_2(s), \widehat{Q}^{(t)}(\zeta_1,\mu_1,\nu_1)(s) \rangle ds.
  \end{aligned}
\end{equation*}
\end{defn}

\section{Main result}\label{main}
\indent\indent
This section is dedicated to deriving the integral-type second order necessary condition for \textbf{Problem (QOP)} using the variational method.

Let $u(\cdot)\in \mathcal{U}^4[0,T]$  be an admissible control that corresponds to the state $x(\cdot)$. Set
\begin{equation}\label{delta-x or delta-u}
\delta u(\cdot)=u(\cdot)-\bar{u}(\cdot)\quad \textrm{and}\quad\delta x(\cdot)=x(\cdot)-\bar{x}(\cdot).
 \end{equation}
For any admissible control $u(\cdot)\in \mathcal{U}[t_0,T]$, let $x_1(\cdot)$ and $x_2(\cdot)$ be the corresponding solution to the following equations:
\begin{equation}\label{the first variation equation}
\left\{
  \begin{aligned}
    dx_1(t)=&\{D_x(t)x_1(t)+ D_u(t)\delta u(t)\}dt+\{F_x(t)x_1(t)+ F_u(t)\delta u(t)\}dW(t)\\
     &\hspace{0.5mm} +dW(t)\{G_x(t)x_1(t)+ G_u(t)\delta u(t)\},\ \textrm{in } [t_0,T], \\
   x_1(t_0)= & 0,
  \end{aligned}
  \right.
\end{equation}
and
\begin{equation}\label{Second chafen equation}
\left\{
\begin{aligned}
dx_2(t)=&\hspace{0.6mm}\{D_x(t)x_2+D_{xx}(t)(x_1,x_1)+2D_{xu}(t)(x_1,\delta u)+D_{uu}(t)(\delta u,\delta u)\}dt\\
&+\{F_x(t)x_2+F_{xx}(t)(x_1,x_1)+2F_{xu}(t)(x_1,\delta u)+F_{uu}(t)(\delta u,\delta u)\}dW(t)\\
&+dW(t)\{G_x(t)x_2+G_{xx}(t)(x_1,x_1)+2G_{xu}(t)(x_1,\delta u)+G_{uu}(t)(\delta u,\delta u)\},\ \textrm{in}\ [t_0,T],\\
x_2(t_0)=&0.
\end{aligned}
\right.
\end{equation}

Combining with It\^{o}-isometry of Clifford stochastic integral in noncommutative space $L^2(\mathscr{C})$,  we have the following auxiliary result.
\begin{lem}\label{estimate}
Let assumptions $\textbf{(H1)}$ and $\textbf{(H3)}$ hold. Then 
\begin{equation}\label{estimate of all QSDEs}
\left\{
 \begin{aligned}
 &\|\delta x(\cdot)\|_{C_\mathbb{A}([t_0,T];L^2(\mathscr{C}))}\leq \mathcal{C}\|\delta u\|_{L^2(t_0,T;\mathcal{H})},\\ 
& \|x_1(\cdot)\|_{C_\mathbb{A}([t_0,T];L^2(\mathscr{C}))}\leq \mathcal{C}\|\delta u\|_{L^2(t_0,T;\mathcal{H})},\\
 & \|x_2(\cdot)\|_{C_\mathbb{A}([t_0,T];L^2(\mathscr{C}))}\leq \mathcal{C}\|\delta u\|^2_{L^4(t_0,T;\mathcal{H})},\\
 &  \|\delta x(\cdot)-x_1(\cdot)\|_{C_\mathbb{A}([t_0,T];L^2(\mathscr{C}))}\leq \mathcal{C}\|\delta u\|^2_{L^4(t_0,T;\mathcal{H})}.
 \end{aligned}
 \right.
\end{equation}
\end{lem}
\begin{proof}
First, we prove the estimate for $\delta x$,  $x_1$ and $x_2$.

For any $\Phi=D,F,G$, and $t\in[t_0,T]$, put
\begin{equation*}
  \left\{
  \begin{aligned}
  \widetilde{\Phi}_{x}(t)&:=\int_0^1\Psi_{x}(t, \bar{x}(t)+\theta \delta x(t),u(t))d\theta,\\
  \widetilde{\Phi}_{u}(t)&:=\int_0^1\Psi_{u}(t, \bar{x}(t), \bar{u}(t)+\theta\delta u(t))d\theta.
  \end{aligned}
  \right.
\end{equation*}
Then, by assumption $\textbf{(H3)}$ and Lemma \ref{estimate of x},  there exists a constant $\mathcal{C}$ such that, for any $t\in[t_0,T]$,
\begin{equation}\label{Psi-x}
\begin{aligned}
\left\| \widetilde{\Phi}_{x}(t)\right\|_{\mathcal{L}(L^2(\mathscr{C}))} &=\left\|\int_0^1\Phi_{x}(t, \bar{x}(t)+\theta \delta x (t),u(t))d\theta\right\|_{\mathcal{L}(L^2(\mathscr{C}))}\\
&\leq\int_0^1\|\Phi_{x}(t, \bar{x}(t)+\theta \delta x (t),u(t))\|_{\mathcal{L}(L^2(\mathscr{C}))}d\theta\\
 &\leq\mathcal{C}.
\end{aligned}
\end{equation}
Similarly, it holds that
\begin{equation}\label{Psi-u}
\left\| \widetilde{\Phi}_{u}(t)\right\|_{\mathcal{L}(\mathcal{H};L^2(\mathscr{C}))}\leq \mathcal{C},\quad t\in[t_0,T].
\end{equation}
By \eqref{delta-x or delta-u}, it is clear that $\delta x$ satisfies the following QSDE:
\begin{equation*}
\left\{
\begin{aligned}
&d\delta x(t)=\left\{\widetilde{D}_x(t)\delta x(t)+\widetilde{D}_u(t)\delta u(t)\right\}dt+\left\{\widetilde{F}_x(t)\delta x(t)+\widetilde{F}_u(t)\delta u(t)\right\}dW(t)\\
&\qquad\qquad+dW(t)\left\{\widetilde{G}_x(t)\delta x(t)+\widetilde{G}_u(t)\delta u(t)\right\},\quad {\rm{in}}\ [{t_0},T],\\
&\delta x({t_0})=0.
\end{aligned}
\right.
\end{equation*}
Then, by  Minkowski inequality, It\^{o}-isometry of Clifford stochastic integral \cite[Theorem 3.15]{B.S.W.1} and \eqref{Psi-x}-\eqref{Psi-u}, we find that, for any $t\in[t_0,T]$,
\begin{align*}
  \|\delta x(t)\|_2^2=&\left\|\int_{t_0}^{t}\left\{\widetilde{D}_x(s)\delta x(s)+\widetilde{D}_u(s)\delta u(s)\right\}ds+\int_{t_0}^{t}\left\{\widetilde{F}_x(s)\delta x(s)+\widetilde{F}_u(s)\delta u(s)\right\}dW(s)\right.\\
   &\hspace{1mm}\left.+\int_{t_0}^{t}dW(s)\left\{\widetilde{G}_x(s)\delta x(s)+\widetilde{G}_u(s)\delta u(s)\right\}\right\|_2^2\\
   \leq &\mathcal{C}\left\{\left\|\int_{t_0}^{t}\left\{\widetilde{D}_x(s)\delta x(s)+\widetilde{D}_u(s)\delta u(s)\right\}ds\right\|_2^2+\int_{t_0}^{t}\left\|\widetilde{F}_x(s)\delta x(s)+\widetilde{F}_u(s)\delta u(s)\right\|_2^2ds\right.\\
   &\hspace{5mm}\left.+ \int_{t_0}^{t}\left\|\widetilde{G}_x(s)\delta x(s)+\widetilde{G}_u(s)\delta u(s)\right\|_2^2 ds\right\}\\
\leq & \mathcal{C}\left\{\int_{t_0}^{t}\left\|\delta x(s)\right\|_2^2 ds+\int_{t_0}^{t}\|\delta u(s)\|_\mathcal{H}^2 ds\right\}.
\end{align*}
By Gronwall's inequality, we obtain that
\begin{equation}\label{the estimate of delta-x}
 \sup_{t\in[t_0,T]}\|\delta x(t)\|_2\leq \mathcal{C}\|\delta u\|_{L^2(t_0,T;\mathcal{H})}.
\end{equation}

From \eqref{the first variation equation}, we have that
\begin{equation*}
\begin{aligned}
 x_1(t)=&\int_{t_0}^t\{D_x(s)x_1(s)+ D_u(s)\delta u(s)\}ds+\int_{t_0}^t\{F_x(s)x_1(s)+ F_u(s)\delta u(s)\}dW(s)\\
  &+\int_{t_0}^tdW(s)\{G_x(s)x_1(s)+ G_u(s)\delta u(s)\}, \ t\in[t_0,T].
  \end{aligned}
\end{equation*}
Similar to the above calculation, one gets that
\begin{equation}\label{the estimate of x_1}
\sup_{t\in[t_0,T]}\|x_1(t)\|_2\leq \mathcal{C}\left(\int_{t_0}^T\|\delta u(s)\|^2_\mathcal{H}ds\right)^{\frac{1}{2}}.
\end{equation}

Next, we prove that $$\|x_2(\cdot)\|_{C_\mathbb{A}([t_0,T];L^2(\mathscr{C}))}\leq \mathcal{C}\|\delta u\|^2_{L^4(t_0,T;\mathcal{H})}.$$ 
By \eqref{Second chafen equation}, \eqref{the estimate of x_1} and  assumption $\textbf{(H3)}$, one obtains that
  \begin{align*}
 &\|x_2(t)\|_2^2\\
  &\leq\mathcal{C}\left\{\left\|\int_{t_0}^t\{D_x(s)x_2(s)+D_{xx}(s)(x_1(s),x_1(s))+2D_{xu}(s)(x_1(s),\delta u(s))\}ds\right\|_2^2\right. \\
  &\quad\indent+\left\|\int_{t_0}^tD_{uu}(s)(\delta u(s),\delta u(s))ds\right\|_2^2+\left\|\int_{t_0}^{t}F_{uu}(s)(\delta u(s),\delta u(s))dW(s)\right\|_2^2\\
  &\quad\indent+\left\|\int_{t_0}^t\{F_x(s)x_2(s)+F_{xx}(s)(x_1(s),x_1(s))+2F_{xu}(s)(x_1(s),\delta u(s))\}dW(s)\right\|_2^2 \\
 &\quad\indent+\left\|\int_{t_0}^t dW(s)\{G_x(s)x_2(s)+G_{xx}(s)(x_1(s),x_1(s))+2G_{xu}(s)(x_1(s),\delta u(s))\}\right\|_2^2\\
 &\quad\indent\left.+\left\|\int_{t_0}^t dW(s)G_{uu}(s)(\delta u(s),\delta u(s))\right\|_2^2\right\} \\
 &\leq\mathcal{C}\left\{ \int_{t_0}^t\|x_2(s)\|_2^2ds+\int_{t_0}^t\|x_1(s)\|_2^4ds+\int_{t_0}^t\|x_1(s)\|_2^2\|\delta u(s)\|^2_{\mathcal{H}}ds+\int_{t_0}^t\|\delta u(s)\|_\mathcal{H}^4ds\right\}\\
 &\leq\mathcal{C}\left\{ \int_{t_0}^t\|x_2(s)\|_2^2ds+\sup_{t\in[t_0,T]}\|x_1(t)\|_2^4+\int_{t_0}^t\|\delta u(s)\|^2_{\mathcal{H}}ds\sup_{t\in[t_0,T]}\|x_1(t)\|_2^2+\int_{t_0}^t\|\delta u(s)\|_\mathcal{H}^4ds\right\}\\
 &\leq\mathcal{C}\left\{ \int_{t_0}^t\|x_2(s)\|_2^2ds+\int_{t_0}^t\|\delta u(s)\|_\mathcal{H}^4ds\right\},\ t\in[t_0,T].
\end{align*}
By  Gronwall's inequality again, we have that
\begin{equation}\label{the estimate of x_2}
\sup_{t\in[t_0,T]}\|x_2(t)\|_2\leq \mathcal{C}\|\delta u\|_{L^4(t_0,T;\mathcal{H})}^2.
\end{equation}

Finally, we prove that 
$$ \|\delta x(\cdot)-x_1(\cdot)\|_{C_\mathbb{A}([t_0,T];L^2(\mathscr{C}))}\leq \mathcal{C}\|\delta u\|^2_{L^4(t_0,T;\mathcal{H})}.$$ 
Let $\xi(\cdot):=\delta x(\cdot)-x_1(\cdot)$. Then $\xi(\cdot)$ satisfies the following QSDE:
\begin{equation*}
\left\{
 \begin{aligned}
 d\xi(t)=&\left\{\widetilde{D}_x(t)\xi(t)+\left(\widetilde{D}_x(t)-D_x(t)\right)x_1(t)+\left(\widetilde{D}_u(t)-D_u(t)\right)\delta u(t)\right\}dt\\
 &+\left\{\widetilde{F}_x(t)\xi(t)+\left(\widetilde{F}_x(t)-F_x(t)\right)x_1(t)+\left(\widetilde{F}_u(t)-F_u(t)\right)\delta u(t)\right\}dW(t)\\
 &+dW(t)\left\{\widetilde{G}_x(t)\xi(t)+\left(\widetilde{G}_x(t)-G_x(t)\right)x_1(t)+\left(\widetilde{G}_u(t)-G_u(t)\right)\delta u(t)\right\},\ {\rm{in}}\ [t_0,T],\\
 \xi(t_0)=&0.
 \end{aligned}
 \right.
\end{equation*}
By assumption $\textbf{(H3)}$,  for any $t\in[t_0,T]$, it follows that 
\begin{equation*}
  \begin{aligned}
  &\|\widetilde{\Phi}_x(t)-\Phi_x(t)\|_{\mathcal{L}(L^2(\mathscr{C}))}\\
  &\hspace{7mm}=\left\|\int_0^1\Phi_{x}(t, \bar{x}(t)+\theta \delta x(t),u(t))d\theta -\Phi_x(t)\right\|_{\mathcal{L}(L^2(\mathscr{C}))} \\
  &\hspace{7mm} \leq\int_0^1\left\|\Phi_{x}(t, \bar{x}(t)+\theta \delta x(t),u(t))-\Phi_x(t)\right\|_{\mathcal{L}(L^2(\mathscr{C}))} d\theta \\
  &\hspace{7mm}\leq \mathcal{C}(\| \delta x (t)\|_2+\| \delta u(t)\|_{\mathcal{H}}).
\end{aligned}
\end{equation*}
Similarly, 
\begin{equation*}
\left\|\widetilde{\Phi}_u(t)-\Phi_u(t)\right\|_{\mathcal{L}(\mathcal{H};L^2(\mathscr{C}))}\leq \mathcal{C}\|\delta u(t)\|_{\mathcal{H}},\ t\in[t_0,T].
\end{equation*}
Hence, we can show from \eqref{the estimate of delta-x}, \eqref{the estimate of x_1}, and \eqref{the estimate of x_2} that
\begin{align*}
  & \|\xi(t)\|_2^2\\
  &\leq\mathcal{C}\left\{\left\|\int_{t_0}^{t}\left\{\widetilde{D}_x(s)\xi(s)+\left(\widetilde{D}_x(s)-D_x(s)\right)x_1(s)+\left(\widetilde{D}_u(s)-D_u(s)\right)\delta u(s)\right\}ds\right\|_2^2 \right. \\
 &\hspace{11mm}+\left\|\int_{t_0}^{t}\left\{\widetilde{F}_x(s)\xi(s)+\left(\widetilde{F}_x(s)-F_x(s)\right)x_1(s)+\left(\widetilde{F}_u(s)-F_u(s)\right)\delta u(s)\right\}dW(s)\right\|_2^2\\
 &\hspace{11mm}\left.+\left\|\int_{t_0}^{t}dW(s)\left\{\widetilde{G}_x(s)\xi(s)+\left(\widetilde{G}_x(s)-G_x(s)\right)x_1(s)+\left(\widetilde{G}_u(s)-G_u(s)\right)\delta u(s)\right\}\right\|_2^2\right\}\\
& \leq\mathcal{C}\left\{\int_{t_0}^{t}\|\xi(s)\|_2^2ds+\int_{t_0}^{t}\left(\|\delta x (s)\|_2^2+\|\delta u(s)\|^2_{\mathcal{H}}\right)\|x_1(s)\|_2^2ds+\int_{t_0}^{t}\|\delta u(s)\|^4_\mathcal{H}ds\right\}\\
 & \leq\mathcal{C}\left\{\int_{t_0}^{t}\|\xi(s)\|_2^2ds+\int_{t_0}^{t}\|\delta u(s)\|^4_\mathcal{H}ds\right\},\ t\in[t_0,T].
\end{align*}
This, together with Gronwall's inequality, implies that
\begin{equation*}
\sup_{t\in[t_0,T]}\|\xi(t)\|_2\leq\mathcal{C} \|\delta u\|_{L^4(t_0,T;\mathcal{H})}^2.
\end{equation*}
This completes the proof of Lemma \ref{estimate}.
\end{proof}
The main result of this paper is articulated as follows.
\begin{thm}\label{second order necessary condition for optimal controls}
Let assumptions $\textbf{(H1)-(H4)}$ hold, and let $\bar{u}(\cdot)\in  \mathcal{U}^4[t_0,T]$ be an optimal control with  $\bar{x}(\cdot)$  as the corresponding optimal state for $\textbf{Problem (QOP)}$.
Then, for the relaxed transposition solution $(P(\cdot),Q^{(\cdot)},\widehat{Q}^{(\cdot)})$ to the equation \eqref{BQSDE-P},  and for any $u(\cdot)\in \mathcal{U}^4[t_0,T]$ with
\begin{equation}\label{condition}
  \int_{t_0}^{T}\langle\mathbb{H}_u(t), u(t)-\bar{u}(t) \rangle_{\mathcal{H}}dt=0,
\end{equation}
the following second order necessary condition holds:
{\small
\begin{align}
  {\rm Re}&\Bigg\{\int_{t_0}^{T}\Big\{\langle\mathbb{H}_{uu}(t)(u(t)-\bar{u}(t)), u(t)-\bar{u}(t) \rangle_{\mathcal{H}}\nonumber\\
  &\hspace{10mm}+\langle(F_x(t)+\Upsilon G_x(t))^*P(t)(F_x(t)+\Upsilon G_x(t)) (u(t)-\bar{u}(t)), u(t)-\bar{u}(t)\rangle_{\mathcal{H}}\nonumber\\
&\hspace{10mm}+2\langle \left(\mathbb{H}_{xu}(t)+D_u(t)^*P(t)+(F_x(t)+\Upsilon G_x(t))^*P(t)(F_x(t)+\Upsilon G_x(t))\right)x_1(t),  u(t)-\bar{u}(t) \rangle_{\mathcal{H}} \label{result}\\
&\hspace{10mm}+\left\langle Q^{(t_0)}(0,D_u(t)(u(t)-\bar{u}(t)),(F_x(t)+\Upsilon G_x(t))(u(t)-\bar{u}(t))),(F_x(t)+\Upsilon G_x(t))(u(t)-\bar{u}(t))\right\rangle\nonumber\\
 &\hspace{10mm}+\left\langle(F_x(t)+\Upsilon G_x(t))(u(t)-\bar{u}(t)),\widehat{Q}^{(t_0)}(0,D_u(t)(u(t)-\bar{u}(t)),(F_x(t)+\Upsilon G_x(t))(u(t)-\bar{u}(t)))\right\rangle \Big\}dt\Bigg\}\nonumber\\
\leq &\hspace{1mm}0.\nonumber
 \end{align}}
\end{thm}
\begin{proof}
Due to the length of the proof, we will present it in several steps.

\textbf{Step 1.} 
Firstly, by definition, it is obvious that $\delta u(\cdot)=u(\cdot)-\bar{u}(\cdot)\in L^4(t_0,T;\mathcal{H})$.
By the convexity of $U$, we have that, for any $\varepsilon\in[0,1]$,
\begin{equation*}
  u^\varepsilon(\cdot):=\bar{u}(\cdot)+\varepsilon\delta u(\cdot)=\bar{u}(\cdot)+\varepsilon(u(\cdot)-\bar{u}(\cdot))=(1-\varepsilon)\bar{u}(\cdot)+\varepsilon u(\cdot)\in L^4(t_0,T;\mathcal{H}).
\end{equation*}
Denote by $x^\varepsilon(\cdot)$ the corresponding solution to the following equation with $u^\varepsilon(\cdot)$, that is,
\begin{equation}\label{FSDE-xvarepsilon}
\left\{
\begin{aligned}
 &dx^\varepsilon(t)=D(t,x^\varepsilon(t),u^\varepsilon(t))dt+F(t,x^\varepsilon(t),u^\varepsilon(t))dW(t)+dW(t)G(t,x^\varepsilon(t),u^\varepsilon(t)),\ \textrm{in}\ [t_0,T],\\
  & x^\varepsilon(t_0)=x_0.
\end{aligned}
\right.
\end{equation}

Let $\delta x^\varepsilon(\cdot):=x^\varepsilon(\cdot)-\bar{x}(\cdot)$.
It follows from Lemma \ref{estimate} that 
\begin{equation}\label{two estimate of delta x-x-1 and itself}
\begin{aligned}
  \|\delta x^\varepsilon(\cdot)\|_{C_\mathbb{A}([t_0,T];L^2(\mathscr{C}))}\leq&\hspace{1mm} \mathcal{C}\varepsilon\|\delta u\|_{L^2(t_0,T;\mathcal{H})},\\
 \|\delta x^\varepsilon(\cdot)-\varepsilon x_1(\cdot)\|_{C_\mathbb{A}([t_0,T];L^2(\mathscr{C}))}&\leq \mathcal{C}\varepsilon^2\|\delta u\|^2_{L^4(t_0,T;\mathcal{H})}.
 \end{aligned}
\end{equation}
For any $\Phi=D,F,G,L$ and $t\in[t_0,T]$, put
  \begin{align*}
  \Phi_{xx}^\varepsilon(t)&:=\int_0^1(1-\theta)\Phi_{xx}(t, \bar{x}(t)+\theta \delta x^\varepsilon(t), \bar{u}(t)+\theta \varepsilon \delta u(t))d\theta,\\
  \Phi_{xu}^\varepsilon(t)&:=\int_0^1(1-\theta)\Phi_{xu}(t, \bar{x}(t)+\theta \delta x^\varepsilon(t), \bar{u}(t)+\theta \varepsilon \delta u(t))d\theta,
  \end{align*}
  and 
  $$\Phi_{uu}^\varepsilon(t):=\int_0^1(1-\theta)\Phi_{uu}(t, \bar{x}(t)+\theta \delta x^\varepsilon(t), \bar{u}(t)+\theta \varepsilon \delta u(t))d\theta.$$
Besides, we define
\begin{equation*}
  g_{xx}^\varepsilon(\bar{x}(T)):=\int_0^1(1-\theta)g_{xx}(\bar{x}(T)+\theta \delta x^\varepsilon(T))d\theta.
\end{equation*}
It is easy to verify that
\begin{equation}\label{First chafen equation}
\left\{
\begin{aligned}
d\delta x^\varepsilon(t)=&\Big\{D_x(t)\delta x^\varepsilon(t)+\varepsilon D_u(t)\delta u(t)+D_{xx}^\varepsilon(t)\left(\delta x^\varepsilon(t),\delta x^\varepsilon(t)\right)\\
&\hspace{2mm}+2\varepsilon D_{xu}^\varepsilon(t)(\delta x^\varepsilon(t),\delta u(t))+\varepsilon^2D_{uu}^\varepsilon(t)(\delta u(t),\delta u(t))\Big\}dt\\
&+\Big\{F_x(t)\delta x^\varepsilon(t)+\varepsilon F_u(t)\delta u(t)+F_{xx}^\varepsilon(t)(\delta x^\varepsilon(t),\delta x^\varepsilon(t))\\
&\hspace{6mm}+2\varepsilon F_{xu}^\varepsilon(t)(\delta x^\varepsilon(t),\delta u(t))+\varepsilon^2 F_{uu}^\varepsilon(t)(\delta u(t),\delta u(t))\Big\}dW(t)\\
&+dW(t)\Big\{G_x(t)\delta x^\varepsilon(t)+\varepsilon G_u(t)\delta u(t)+G_{xx}^\varepsilon(t)(\delta x^\varepsilon(t),\delta x^\varepsilon(t))\\
&\hspace{16mm}+2\varepsilon G_{xu}^\varepsilon(t)(\delta x^\varepsilon(t),\delta u(t))+\varepsilon^2 G_{uu}^\varepsilon(t)(\delta u(t),\delta u(t))\Big\},\ {\rm{in}}\ [t_0,T],\\
\delta x^\varepsilon(t_0)=&0.
\end{aligned}
\right.
\end{equation}

\textbf{Step 2.} We claim that there exists a sequence $\{\varepsilon_n\}_{n=1}^\infty\subset (0,1]$ such that $\lim\limits_{n\to\infty}\varepsilon_n=0$, and 
\begin{equation}\label{the key estimate}
\left\|\delta x^{\varepsilon_n}(\cdot)-\varepsilon_nx_1(\cdot)-\frac{\varepsilon_n^2}{2}x_2(\cdot)\right\|_{C_\mathbb{A}([t_0,T];L^2(\mathscr{C}))}=o(\varepsilon_n^2), \textrm{ as n}\to\infty.
\end{equation}
For this purpose, we write
\begin{equation}
z^\varepsilon(\cdot):=\varepsilon^{-2}\left(\delta x^\varepsilon(\cdot)-\varepsilon x_1(\cdot)-\frac{\varepsilon^2}{2}x_2(\cdot)\right),
\end{equation}
then $z^\varepsilon(\cdot)$ satisfies
\begin{equation}
\left\{
\begin{aligned}
dz^\varepsilon(t)=&\left\{D_x(t)z^\varepsilon(t)+D_{xx}^\varepsilon(t)\left(\frac{\delta x^\varepsilon(t)}{\varepsilon},\frac{\delta x^\varepsilon(t)}{\varepsilon}\right)-\frac{1}{2}D_{xx}(t)(x_1(t),x_1(t))\right.\\
&\hspace{3mm}+2D_{xu}^\varepsilon(t) \Big(\frac{\delta x^\varepsilon(t)}{\varepsilon},\delta u(t)\Big)-D_{xu}(t)(x_1(t),\delta u(t))\\
&\hspace{3mm}\left.+D^\varepsilon_{uu}(t)(\delta u(t),\delta u(t))-\frac{1}{2}D_{uu}(t)(\delta u(t),\delta u(t))\right\}dt\\
&+\left\{F_x(t)z^\varepsilon(t)+F_{xx}^\varepsilon(t)\left(\frac{\delta x^\varepsilon(t)}{\varepsilon},\frac{\delta x^\varepsilon(t)}{\varepsilon}\right)-\frac{1}{2}F_{xx}(t)\left(x_1(t),x_1(t)\right)\right.\\
&\hspace{7mm}+2F_{xu}^\varepsilon(t) \left(\frac{\delta x^\varepsilon(t)}{\varepsilon},\delta u(t)\right)-F_{xu}(t)(x_1(t),\delta u(t))\\
&\hspace{7mm}\left.+F^\varepsilon_{uu}(t)(\delta u(t),\delta u(t))-\frac{1}{2}F_{uu}(t)(\delta u(t),\delta u(t))\right\}dW(t)\\
&+dW(t)\left\{G_x(t)z^\varepsilon(t)+G_{xx}^\varepsilon(t)\left(\frac{\delta x^\varepsilon(t)}{\varepsilon},\frac{\delta x^\varepsilon(t)}{\varepsilon}\right)-\frac{1}{2}G_{xx}(t)(x_1(t),x_1(t))\right.\\
&\hspace{18mm}+2G_{xu}^\varepsilon(t) \left(\frac{\delta x^\varepsilon(t)}{\varepsilon},\delta u(t)\right)-G_{xu}(t)(x_1(t),\delta u(t))\\
&\hspace{18mm}\left.+G^\varepsilon_{uu}(t)(\delta u(t),\delta u(t))-\frac{1}{2}G_{uu}(t)(\delta u(t),\delta u(t))\right\},\ {\rm{in}}\ [t_0,T],\\
z^\varepsilon(t_0)=&0.
\end{aligned}
\right.
\end{equation}
For any $t\in[t_0,T]$, let
\begin{align*}
  \Phi_{D,\varepsilon}(t):= &\hspace{1mm}D_{xx}^\varepsilon(t)\left(\frac{\delta x^\varepsilon(t)}{\varepsilon},\frac{\delta x^\varepsilon(t)}{\varepsilon}\right)-\frac{1}{2}D_{xx}(t)(x_1(t),x_1(t))+2D_{xu}^\varepsilon(t) \left(\frac{\delta x^\varepsilon(t)}{\varepsilon},\delta u(t)\right) \\
  &\hspace{1mm} -D_{xu}(t)(x_1(t),\delta u(t))+D^\varepsilon_{uu}(t)(\delta u(t),\delta u(t))-\frac{1}{2}D_{uu}(t)(\delta u(t),\delta u(t)) ,\\
   \Phi_{F,\varepsilon}(t):= &\hspace{1mm}F_{xx}^\varepsilon(t)\left(\frac{\delta x^\varepsilon(t)}{\varepsilon},\frac{\delta x^\varepsilon(t)}{\varepsilon}\right)-\frac{1}{2}F_{xx}(t)(x_1(t),x_1(t))+2F_{xu}^\varepsilon(t) \left(\frac{\delta x^\varepsilon(t)}{\varepsilon},\delta u(t)\right) \\
  & \hspace{1mm}-F_{xu}(t)(x_1(t),\delta u(t))+F^\varepsilon_{uu}(t)(\delta u(t),\delta u(t))-\frac{1}{2}F_{uu}(t)(\delta u(t),\delta u(t)), \\
   \Phi_{G,\varepsilon}(t):= &\hspace{1mm}G_{xx}^\varepsilon(t)\left(\frac{\delta x^\varepsilon(t)}{\varepsilon},\frac{\delta x^\varepsilon(t)}{\varepsilon}\right)-\frac{1}{2}G_{xx}(t)(x_1(t),x_1(t))+2G_{xu}^\varepsilon(t) \left(\frac{\delta x^\varepsilon(t)}{\varepsilon},\delta u(t)\right) \\
  & \hspace{1mm}-G_{xu}(t)(x_1(t),\delta u(t))+G^\varepsilon_{uu}(t)(\delta u(t),\delta u(t))-\frac{1}{2}G_{uu}(t)(\delta u(t),\delta u(t)).
\end{align*}
Then, 
\begin{equation}\label{the integral-type of r_2}
\begin{aligned}
z^\varepsilon(t)=&\int_{t_0}^t\left\{D_x(s)z^\varepsilon(s)+\Phi_{D,\varepsilon}(s)\right\}ds+\int_{t_0}^t\left\{F_x(s)z^\varepsilon(s)+\Phi_{F,\varepsilon}(s)\right\}dW(s)\\
&+\int_{t_0}^t dW(s)\left\{G_x(s)z^\varepsilon(s)+\Phi_{G,\varepsilon}(s)\right\},\ t\in[t_0,T].
\end{aligned}
\end{equation}
Therefore, we deduce that for any $t\in[t_0,T]$,
\begin{equation}\label{estimate of r_2}
 \|z^\varepsilon(t)\|_2^2\leq\mathcal{C}\left\{ \int_{t_0}^t\|z^\varepsilon(s)\|^2_2ds+\int_{t_0}^{t}\{\|\Phi_{D,\varepsilon}(s)\|^2_2+\|\Phi_{F,\varepsilon}(s)\|^2_2+\|\Phi_{G,\varepsilon}(s)\|^2_2\}ds\right\}.
\end{equation}
By means of $$\|\delta x^\varepsilon(\cdot)\|_{C_\mathbb{A}([{t_0},T];L^2(\mathscr{C}))}\leq \mathcal{C}\varepsilon \|\delta u \|_{L^2({t_0},T;\mathcal{H})},$$ 
we can find that there exists a subsequence $\{\varepsilon_n\}_{n=1}^\infty\subset (0,1]$ such that $\varepsilon_n\to 0$ and $x^{\varepsilon_n}(\cdot)\to \bar{x}(\cdot)$ a.e in $[t_0,T]$, as $n\to\infty$. Then, by  Lebesgue's dominated convergence theorem, we deduce that
\begin{equation}\label{estimate Phi-(D,varepsilon_n)}
  \begin{aligned}
\lim_{n\to\infty} &\int_{t_0}^T\|\Phi_{D,\varepsilon_n}(s)\|_2^2ds \\
 &\leq \lim_{n\to\infty}\int_{t_0}^T\left\|D_{xx}^{\varepsilon_n}(s)\left(\frac{\delta x^{\varepsilon_n}(s)}{\varepsilon_n},\frac{\delta x^{\varepsilon_n}(s)}{\varepsilon_n}\right)-\frac{1}{2}D_{xx}(s)(x_1(s),x_1(s))\right.    \\
  &\hspace{22mm}+2D_{xu}^{\varepsilon_n}(s)\left(\frac{\delta x^{\varepsilon_n}(s)}{\varepsilon_n},\delta u(s)\right)-D_{xu}(s)(x_1(s),\delta u(s))  \\
   &\hspace{22mm}\left.+D^{\varepsilon_n}_{uu}(s)(\delta u(s),\delta u(s))-\frac{1}{2}D_{uu}(s)(\delta u(s),\delta u(s))\right\|_2^2ds  \\
  & \leq \mathcal{C}\lim_{n\to\infty}\int_{t_0}^T\left\{\left\|D_{xx}^{\varepsilon_n}(s)\left\{\left(\frac{\delta x^{\varepsilon_n}(s)}{\varepsilon_n},\frac{\delta x^{\varepsilon_n}(s)}{\varepsilon_n}\right)-(x_1(s),x_1(s))\right\}\right\|_2^2\right. \\
   &\hspace{26mm}+\left\|D_{xx}^{\varepsilon_n}(s)-\frac{1}{2}D_{xx}(s)\right\|^2_{\mathcal{L}(L^2(\mathscr{C}),L^2(\mathscr{C});L^2(\mathscr{C}))}\|x_1(s)\|_{2}^4 \\
   &\hspace{26mm}+2\|D_{xu}^{\varepsilon_n}(s)\|^2_{\mathcal{L}(L^2(\mathscr{C}),\mathcal{H};L^2(\mathscr{C}))}\left\|\frac{\delta x^{\varepsilon_n}(s)}{\varepsilon_n}-x_1(s)\right\|_2^2\|\delta u(s)\|^2_\mathcal{H} \\
   &\hspace{26mm}+\|D_{xu}^{\varepsilon_n}(s)-D_{xu}(s)\|^2_{\mathcal{L}(L^2(\mathscr{C}),\mathcal{H};L^2(\mathscr{C}))}\|x_1(s)\|_2^2\|\delta u(s)\|^2_\mathcal{H} \\
    &\hspace{26mm}\left.+\left\|D^{\varepsilon_n}_{uu}(s)-\frac{1}{2}D_{uu}(s)\right\|_{\mathcal{L}(\mathcal{H},\mathcal{H};L^2(\mathscr{C}))}\|\delta u(s)\|_\mathcal{H}^4\right\} ds  \\
    &=0. 
\end{aligned}
\end{equation}
Similarly,
\begin{equation}\label{estimate Phi-(F,varepsilon_n) and Phi-(G,varepsilon_n)}
 \lim_{n\to\infty}\int_{t_0}^T\|\Phi_{F,\varepsilon_n}(s)\|_2^2ds=0,\quad \lim_{n\to\infty}\int_{t_0}^T\|\Phi_{G,\varepsilon_n}(s)\|_2^2ds=0.
\end{equation}
By combining \eqref{estimate of r_2}-\eqref{estimate Phi-(F,varepsilon_n) and Phi-(G,varepsilon_n)}, along with the Gronwall inequality, we derive   \eqref{the key estimate}.

\textbf{Step 3.} By Taylor's formula \cite[Lemma 3.4]{Z.Z-2017}, we have 
\begin{equation}\label{Taylor of L}
\begin{aligned}
&L(t, x^\varepsilon(t),u^\varepsilon(t))-L(t,\bar{x}(t),\bar{u}(t))\\
&=\langle L_x(t),\delta x^\varepsilon(t) \rangle+\varepsilon\langle L_u(t), \delta u(t) \rangle_\mathcal{H}+\langle L_{xx}^\varepsilon(t)\delta x^\varepsilon(t), \delta x^\varepsilon(t)\rangle\\
&\indent+\langle 2 L_{xu}^\varepsilon(t)\delta x^\varepsilon(t), \varepsilon\delta u(t)\rangle_\mathcal{H}+\varepsilon^2\langle L_{uu}^\varepsilon(t)\delta u(t), \delta u(t)\rangle_\mathcal{H},
\end{aligned}
\end{equation}
and 
\begin{equation}\label{Taylor of g}
g( x^\varepsilon(T))-g(\bar{x}(T))=\langle g_x(\bar{x}(T)),\delta x^\varepsilon(T) \rangle+\langle g_{xx}^\varepsilon(T)\delta x^\varepsilon(T), \delta x^\varepsilon(T)\rangle.
\end{equation}
Similar to the proof of \eqref{the key estimate}, we can show that for the subsequence $\{\varepsilon_n\}_{n=1}^\infty\subset (0,1]$ such that $x^{\varepsilon_n}(\cdot)\to \bar{x}(\cdot)$ a.e. in $[t_0,T]$.
By means of \eqref{two estimate of delta x-x-1 and itself}, we have that
\begin{align*}
  &\lim_{n\to\infty}\frac{1}{\varepsilon_n^2}\left\{\langle g_{xx}^{\varepsilon_n}(\bar{x}(T))\delta x^{\varepsilon_n}(T), \delta x^{\varepsilon_n}(T)\rangle-\frac{\varepsilon_n^2}{2}\langle g_{xx}(\bar{x}(T))x_1(T),x_1(T) \rangle\right\} \\
 &\hspace{1mm}= \lim_{n\to\infty}\frac{1}{\varepsilon_n^2}\left\{\left\langle \left(g_{xx}^{\varepsilon_n}(\bar{x}(T))-\frac{1}{2} g_{xx}(\bar{x}(T))\right)\delta x^{\varepsilon_n}(T), \delta x^{\varepsilon_n}(T)\right\rangle\right.\\
 &\hspace{1mm}\indent\indent\indent\quad+\left\langle\frac{1}{2} g_{xx}(\bar{x}(T))(\delta x^{\varepsilon_n}(T)-\varepsilon_n x_1(T)),\delta x^{\varepsilon_n}(T) \right\rangle\\
&\hspace{1mm}\indent\indent\indent\quad\left. +\left\langle\frac{1}{2} g_{xx}(\bar{x}(T))\varepsilon_n x_1(T), (\delta x^{\varepsilon_n}(T)-\varepsilon_n x_1(T))\right\rangle\right\}\\
&\hspace{1mm}=0.
\end{align*}
Similarly, one gets that
\begin{gather*} 
\lim_{n\to\infty}\int_{t_0}^T\left\{\langle L_{uu}^{\varepsilon_n}(t)\delta u(t), \delta u(t)\rangle_\mathcal{H}-\frac{1}{2}\langle L_{xx}(t)\delta u(t),\delta u(t) \rangle_\mathcal{H}  \right\}dt=0,\\ \lim_{n\to\infty}\frac{1}{\varepsilon_n^2}\int_{t_0}^T\left\{\langle L_{xx}^{\varepsilon_n}(t)\delta x^{\varepsilon_n}(t), \delta x^{\varepsilon_n}(t)\rangle-\frac{\varepsilon_n^2}{2}\langle L_{xx}(t)x_1(t),x_1(t) \rangle\right\}dt=0,
\end{gather*} 
and
\begin{equation*}
  \lim_{n\to\infty}\frac{1}{\varepsilon_n^2}\int_{t_0}^T\left\{\langle L_{xu}^{\varepsilon_n}(t)\delta x^{\varepsilon_n}(t), \varepsilon_n\delta u(t)\rangle_\mathcal{H}-\frac{\varepsilon_n^2}{2}\langle L_{xu}(t)x_1(t),\delta u(t)\rangle_\mathcal{H}\right\}dt=0.
\end{equation*}
These, together with \eqref{the key estimate}, \eqref{Taylor of L} and \eqref{Taylor of g}, imply that
  \begin{align}
    & \mathcal{J}(u^{\varepsilon_n}(\cdot))-\mathcal{J}(\bar{u}(\cdot))\nonumber\\
    &=\textrm{Re}\Bigg\{\langle g_x(\bar{x}(T)),\delta x^\varepsilon(T) \rangle+\langle g_{xx}^\varepsilon(T)\delta x^\varepsilon(T), \delta x^\varepsilon(T)\rangle\nonumber\\
     &\hspace{1.2cm}+\int_{t_0}^T\Big\{\langle L_x(t),\delta x^\varepsilon(t) \rangle+\varepsilon\langle L_u(t), \delta u(t) \rangle_\mathcal{H}+\langle L_{xx}^\varepsilon(t)\delta x^\varepsilon(t), \delta x^\varepsilon(t)\rangle \nonumber\\
&\hspace{2.5cm} +\langle 2 L_{xu}^\varepsilon(t)\delta x^\varepsilon(t), \varepsilon\delta u(t)\rangle_\mathcal{H}+\varepsilon^2\langle L_{uu}^\varepsilon(t)\delta u(t), \delta u(t)\rangle_\mathcal{H}\Big\}dt \Bigg\}\label{zuocha of cost functional}\\
     & =\textrm{Re}\Bigg\{\varepsilon_n\langle g_x(\bar{x}(T)),x_1(T)\rangle+\frac{\varepsilon_n^2}{2}\langle g_x(\bar{x}(T)),x_2(T)\rangle+\frac{\varepsilon_n^2}{2}\langle g_{xx}(\bar{x}(T))x_1(T),x_1(T)\rangle\nonumber\\
     &\hspace{1.2cm}+\int_{t_0}^T\bigg\{\frac{\varepsilon_n^2}{2}\Big[\langle L_{xx}(t)x_1(t), x_1(t)\rangle+2\langle L_{xu}(t)x_1(t), \delta u(t)\rangle_\mathcal{H}+\langle L_{uu}(t)\delta u(t), \delta u(t)\rangle_\mathcal{H}\Big]\nonumber\\
    & \hspace{2.6cm}+\varepsilon_n\langle L_x(t),x_1(t)\rangle+\frac{\varepsilon_n^2}{2} \langle L_x(t),x_2(t)\rangle+\varepsilon_n\langle L_u(t),\delta u(t)\rangle_\mathcal{H}\bigg\} dt\Bigg\}\nonumber\\
    & \hspace{5.5mm}+\textit{\textbf{o}}(\varepsilon_n^2),\quad \textrm{as}\  n\to\infty.\nonumber
  \end{align}

\textbf{Step 4.} By applying Fermion It\^{o}'s formula \cite[Theorem 5.2]{A.H-2} to $\langle y(t), x_1(t)\rangle$, we have that
\begin{align}\label{y-x_1}
&-\langle g_x(\bar{x}(T)), x_1(T)\rangle\nonumber\\
&=\int_{t_0}^{T}\Big\{\langle y(t), D_x(t)x_1(t)\rangle+\langle y(t),D_u(t)\delta u(t) \rangle\Big\} dt\nonumber\\
&\hspace{5mm}+\int_{t_0}^{T}\langle Y(t), (F_x(t)+\Upsilon G_x(t))x_1(t)+(F_u(t)+\Upsilon G_u(t))\delta u(t)\rangle dt\\
&\hspace{5mm}-\int_{t_0}^{T}\langle D_x(t)^*y(t)+(F_x(t)+\Upsilon G_x(t))^*Y(t), x_1(t)\rangle dt+\int_{t_0}^{T}\langle L_x(t),x_1(t)\rangle dt\nonumber\\
&=\int_{t_0}^{T}\Big\{\langle y(t),D_u(t)\delta u(t) \rangle+\langle L_x(t),x_1(t)\rangle+\langle Y(t), (F_u(t)+\Upsilon G_u(t))\delta u(t)\rangle\Big\} dt\nonumber.
\end{align}
Similar to \eqref{y-x_1}, we also derive that 
\begin{align}\label{y-x_2}
-&\langle g_x(\bar{x}(T)), x_2(T)\rangle\nonumber\\
=&\int_{t_0}^{T}\Big\{\langle L_x(t),x_2(t)\rangle+\langle y(t),D_{xx}(t)(x_1(t),x_1(t))\rangle+\langle y(t),D_{uu}(t)(\delta u(t),\delta u(t))\rangle\nonumber\\
&\hspace{9mm}+\langle y(t),2D_{xu}(t)(x_1(t),\delta u(t))\rangle+\langle Y(t), 2(F_{xu}(t)+\Upsilon G_{xu}(t))(x_1(t),\delta u(t))\rangle\\
&\hspace{9mm}+\langle Y(t), (F_{xx}(t)+\Upsilon G_{xx}(t))(x_1(t),x_1(t))\rangle+\langle Y(t), (F_{uu}(t)+\Upsilon G_{uu}(t))(\delta u(t),\delta u(t))\rangle\Big\} dt\nonumber.
\end{align}
By \cite[Theorem 1.2]{W.W-2}, the triple $\left(P(\cdot),Q^{(\cdot)},\widehat{Q}^{(\cdot)}\right)$ is the relaxed transposition to \eqref{BQSDE-P} in the sense of  Definition \ref{definition of the relaxed tran solution}, then
\begin{equation}\label{P-x_1}
 \begin{aligned}
   -&\langle g_{xx}(\bar{x}(T)) x_1(T), x_1(T)\rangle +\int_{t_0}^{T}\langle\mathbb{H}_{xx}(t)x_1(t),x_1(t)\rangle dt\\
  &\quad =\int_{t_0}^{T}\left\{\langle P(t)x_1(t), D_u(t)\delta u(t)\rangle+\langle P(t)D_u(t)\delta u(t), x_1(t)\rangle\right\} dt \\
   &\qquad+\int_{t_0}^{T}\langle P(t)(F_{x}(t)+\Upsilon G_{x}(t))x_1(t), (F_{u}(t)+\Upsilon G_{u}(t))\delta u(t)\rangle dt\\
   &\qquad+\int_{t_0}^{T}\langle P(t)(F_{u}(t)+\Upsilon G_{u}(t))\delta u(t), (F_{u}(t)+\Upsilon G_{u}(t))\delta u(t)\rangle dt  \\
    & \qquad+\int_{t_0}^{T}\langle P(t)(F_{u}(t)+\Upsilon G_{u}(t))\delta u(t), (F_{x}(t)+\Upsilon G_{x}(t))x_1(t)\rangle dt\\
     &\qquad+\int_{t_0}^T\left\langle Q^{(t_0)}(0,D_u(t)\delta u,(F_{u}(t)+\Upsilon G_{u}(t))\delta u)(t),  (F_{u}(t)+\Upsilon G_{u}(t))\delta u(t)\right\rangle dt\\
     &\qquad+\int_{t_0}^T\left\langle(F_{u}(t)+\Upsilon G_{u}(t))\delta u(t), \widehat{Q}^{(t_0)}(0,D_u(t)\delta u,(F_{u}(t)+\Upsilon G_{u}(t))\delta u)(t)\right\rangle dt.
 \end{aligned}
\end{equation}
By substituting \eqref{y-x_1}-\eqref{P-x_1} into \eqref{zuocha of cost functional}, we deduce that
\begin{align*}
0\leq &\hspace{1mm}\frac{\mathcal{J}(u^{\varepsilon_n}(\cdot))-\mathcal{J}(\bar{u}(\cdot))}{\varepsilon_n^2}\nonumber\\
=&\hspace{0.5mm}\textrm{Re}\Bigg\{\int_{t_0}^T\left\{\frac{1}{\varepsilon_n}\langle L_x(t),x_1(t)\rangle+\frac{1}{2}\langle L_x(t),x_2(t)\rangle+\frac{1}{\varepsilon_n}\langle L_u(t),\delta u(t)\rangle_\mathcal{H}\right.\nonumber\\
&\hspace{17mm}\left.+\frac{1}{2}\langle L_{xx}(t)x_1(t),x_1(t)\rangle+\langle L_{xu}(t)x_1(t),\delta u(t)\rangle_\mathcal{H}+\frac{1}{2}\langle L_{uu}(t)\delta u(t),\delta u(t)\rangle_\mathcal{H}\right\}dt\nonumber\\
&\hspace{9mm}+\frac{1}{\varepsilon_n}\langle g_x(\bar{x}(T)), x_1(T)\rangle+\frac{1}{2}\langle g_x(\bar{x}(T)), x_2(T)\rangle+\frac{1}{2}\langle g_{xx}(\bar{x}(T))x_1(T), x_1(T)\rangle\Bigg\}+\textit{\textbf{o}}(1)\nonumber\\
=&\hspace{0.5mm}\textrm{Re}\Bigg\{-\frac{1}{2}\int_{t_0}^T\left\{\langle \mathbb{H}_{uu}(t)\delta u(t),\delta u(t)\rangle_\mathcal{H}+\langle(F_{u}(t)+\Upsilon G_{u}(t))^*P(t)(F_{u}(t)+\Upsilon G_{u}(t))\delta u(t),\delta u(t)\rangle_\mathcal{H}\right\} dt\nonumber\\
&\hspace{8mm}-\int_{t_0}^{T}\langle\mathbb{H}_{xu}(t)x_1(t),\delta u(t)\rangle_\mathcal{H} dt-\int_{t_0}^T\langle D_u(t)^*P(t)x_1(t),\delta u(t)\rangle_\mathcal{H} dt\nonumber\\
&\hspace{8mm}-\int_{t_0}^T\langle(F_{u}(t)+\Upsilon G_{u}(t))^*P(t)(F_{x}(t)+\Upsilon G_{x}(t))x_1(t),\delta u(t)\rangle_\mathcal{H} dt\nonumber\\
&\hspace{8mm}-\frac{1}{2}\int_{t_0}^T\left\{\left\langle Q^{({t_0})}(0,D_u(t)\delta u,(F_{u}(t)+\Upsilon G_{u}(t))\delta u)(t),  (F_{u}(t)+\Upsilon G_{u}(t))\delta u(t)\right\rangle\right.\nonumber\\
&\hspace{25mm}\left.+\left\langle(F_{u}(t)+\Upsilon G_{u}(t))\delta u(t), \widehat{Q}^{({t_0})}(0,D_u(t)\delta u,(F_{u}(t)+\Upsilon G_{u}(t))\delta u)(t)\right\rangle\right\} dt\Bigg\}.
\end{align*}
Letting $n\to\infty$ in the above inequality, we can obtain the desired second order necessary condition \eqref{result}. The proof completes.
\end{proof}

\end{document}